\documentclass[reqno,a4paper,11pt]{amsart}

\usepackage[all,poly]{xy}
\usepackage{amsfonts}
\usepackage[mathcal]{eucal}
\usepackage{eufrak}
\usepackage{amssymb}
\usepackage{amsmath}
\usepackage{mathrsfs}
\usepackage{color}
\usepackage[pagebackref,colorlinks]{hyperref}
\usepackage{enumerate}

\theoremstyle{plain}
\newtheorem {lemma}{Lemma} 
\newtheorem {theorem}[lemma]{Theorem}

\newtheorem {corollary}[lemma]{Corollary}

\theoremstyle{definition}

\newtheorem {example}[lemma]{Example}

\theoremstyle{definition}

{}

\newcommand{\Ker}{\operatorname{Ker}}

\newcommand{\rModd}{\operatorname{-Mod}}

\newcommand{\Hom}{\operatorname{Hom}}

\renewcommand{\Im}{\operatorname{Im}}

\title[Recollements and Leavitt path algebras]{Recollements, sinks elimination and\\   Leavitt path algebras}

\author{R. Hazrat}
\address{
Western Sydney University,
Australia}\email{r.hazrat@westernsydney.edu.au}

\author{J. Huang}
\address{
Fujian Normal University, China
}\email{754068710@qq.com}

\begin{document}
\begin{abstract}
For Leavitt path algebras, we show that whereas removing sources  from a graph produces a Morita equivalence, removing sinks gives rise to a recollement situation. In general, we show that for a graph $E$  and a finite hereditary subset $H$ of $E^0$ there is a recollement 
$$\xymatrix{
L_K(E/\overline H) \rModd  \ar[r] & \ar@<3pt>[l]    \ar@<-3pt>[l]  
L_K(E) \rModd  \ar[r]  & \ar@<3pt>[l] \ar@<-3pt>[l] L_K(E_H) \rModd  .}$$
We record several corollaries. 
\end{abstract}

\maketitle

In this short note we record an application of recollements in naturally decomposing a Leavitt path algebra and gluing the pieces together. 
A recollement (gluing) of abelian categories consists of three abelian categories $\mathscr A, \mathscr B, \mathscr C$ and six functors relating them as follows
\begin{equation}\label{popop}
\xymatrix{
\mathscr A \ar[r]|-{i} & \ar@<5pt>[l]^p    \ar@<-5pt>[l]_q  
\mathscr B \ar[r]|-{j}  & \ar@<5pt>[l]^r \ar@<-5pt>[l]_s   \mathscr C,}
\end{equation}
such that 
\begin{enumerate}
\item[(i)] $(s,j,r)$ and $(q,i,p)$ are adjoint triples, i.e., $s$ is left adjoint to $j$ which is left adjoint to $r$, similarly for the second triple;
\item[(ii)] the functors $i$, $s$, and $r$ are fully faithful;
\item[(iii)] $\Im (i)=\Ker (j)$.
\end{enumerate}

In the setting of unital rings, the following is an archetype example of recollement:  Let $A$ be a ring with identity and $e\in A$ an idempotent. Then there is a recollement situation 
\begin{equation}\label{recolidem}
\xymatrix{
A/AeA \rModd   \ar[rr]|-{inc} && \ar@<5pt>[ll]^{\Hom_A(A/AeA,-)}   \ar@<-5pt>[ll]_{A/AeA \otimes_A -}  
A \rModd  \ar[rr]|-{e(-)}  && \ar@<5pt>[ll]^{\Hom_{eAe}(Ae,-)} \ar@<-5pt>[ll]_{eA\otimes_{eAe} -}   eAe \rModd .}
\end{equation}

To check the adjointness of (i) and the fully faithfullness of (ii), one uses the following general hom-tensor calculus (see~\cite[\S20]{fuller}): For a triple $({}_R M,{}_S W_R, {}_S N)$ of modules over unital rings $R$ and $S$, we have a natural isomorphism 
\[
\Hom_R(M,\Hom_S(W,N)) \cong \Hom_S(W\otimes_R M,N).
\] Furthermore, if $M$ is finitely generated projective $R$-module then 
\[ 
\Hom_S(N,W)\otimes_R M \cong \Hom_S(N,W\otimes_T M).
\]
Finally for the triple  $(M_R,{}_S W_R, {}_S N)$, where $M$  is finitely generated projective $R$-module, we have the natural isomorphism 
\[ M\otimes_R\Hom_S(W,N)\cong \Hom_S(\Hom(M,W),N).\]

For rings with local units, the calculus of module theory is similar to the rings with units. In particular the above isomorphisms are valid if $R$ is a ring with local units and $S$ is a unital ring (see for example~\cite[49.1--49.3]{wisba}). Since $eAe$ is a ring with unit, the recollement~(\ref{recolidem}) can be extended for a ring $A$ with local units. 

The concept of recollement has a geometric origin, and it was first appeared in the work of Beilinson, Bernstein and Delign~\cite{beil} in the setting of triangulated categories and then in the setting of abelian categories in Franjou and Pirashvilli~\cite{fran} motivated by MacPherson-Vilonen construction for the category of perverse sheaves (see~\cite{fran,psaer} for background on recollements, applications and further references). 

In this note we work with a directed graph $E = (E^0,E^1,r,s)$ which consist of two sets $E^0$, $E^1$ and maps $r,s : E^1\rightarrow E^0.$ We consider the Leavitt path algebra $L(E)$ associated to $E$. For a background on Leavitt path algebras and the relevant terminologies see for example~\cite{abrams, ranga2,ranga1} and the references there. 

To state the main theorem we need to recall the notion of restriction and quotient of a graph. Let $E$ be a graph and $H$ a hereditary subset $E^0$. We denote by
$E_H$ the {\it restriction graph}
$$\Big (H, \{e\in E^1\mid s(e)\in H\}, r|_{(E_H)^1},s|_{(E_H)^1}\Big).$$
On the other hand, for a hereditary and saturated subset $X$, we denote by $E/X$ the {\it quotient graph}
$$\Big(E^0\setminus X,\{ e\in E^1\mid r(e)\notin X\}, r|_{(E/X)^1},s|_{(E/X)^1}\Big)$$

We are in a position to record our theorem. 
\begin{theorem}\label{meis}
Let $E$ be a row finite graph, $H$ a finite hereditary subset of $E^0$ and let $\overline H$ be its hereditary saturated closure.  Then there is a recollement of abelian categories 
$$\xymatrix{
L_K(E/\overline H) \rModd  \ar[r] & \ar@<3pt>[l]    \ar@<-3pt>[l]  
L_K(E) \rModd  \ar[r]  & \ar@<3pt>[l] \ar@<-3pt>[l] L_K(E_H)\rModd  .}$$
\end{theorem}
\begin{proof}

Observe that since $H$ is finite 
\begin{equation}\label{idempp}
L(E_H)=p_HL(E)p_H,
\end{equation}
 where
$p_H=\sum _{v\in H} v\in L(E)$ is an idempotent. 

Next we observe that $\langle H\rangle=\langle \overline H\rangle$. The hereditary saturated closure of $H$ is $\overline H =\cup_{n=0}^\infty \Lambda_n(H)$, where $\Lambda_0(H)=H$ and 
$$\Lambda_n(H)=\big \{y \in E^0 \mid s^{-1}(y) \not = \emptyset, r(s^{-1}(y)) \subseteq \Lambda_{n-1}(H)\big \} \cup \Lambda_{n-1}(H),$$ for $n\geq 1$. Clearly  $\langle H\rangle \subseteq \langle \overline H\rangle$.
We prove the converse by induction. For $n=0$, $\Lambda_0(H)=H \subseteq \langle H\rangle$. Suppose $\Lambda_{n-1}(H)\subseteq \langle H\rangle$ and let $y \in \Lambda_{n}(H)$. Then $y = \sum_{\alpha \in s^{-1}(y)} \alpha \alpha^*$. Since $r(\alpha) \in \Lambda_{n-1}(H) \subseteq \langle H\rangle$ and $\langle H\rangle$ is a two-sided ideal, it follows that $y \in  \langle H\rangle$. Thus $\overline H= \cup_{n=0}^\infty \Lambda_n(H) \subseteq \langle H\rangle$.  We have a natural isomorphism
\begin{equation}\label{idempp2}
L(E)/\langle H\rangle= L(E)/ \langle \overline H\rangle\cong L(E/\overline H).
\end{equation}
Now replacing (\ref{idempp}) and (\ref{idempp2}) in recollement diagram~\ref{recolidem}, the theorem follows. 
\end{proof}

It was observed in~\cite{abrams} that removing sources from a finite graph would give a Leavitt path algebra Morita equivalent to the one associated to the original graph. However removing sinks change the structure of the Leavitt path algebras substantially.  As an example, consider the Topilz algebra, i.e.,  the Leavitt path algebra associated to the graph  
\begin{equation*}
\xymatrix{
E: \, \, & \bullet \ar@(ld,lu)^{e}  \ar[r]^{f}  &  \bullet .  
}
\end{equation*}
As soon as we remove the sink and the edge $f$ attached to it, we are left with a loop whose  Leavitt path algebra is the Laurent polynomial ring $K[x,x^{-1}]$.  Whereas removing sources gives a Morita equivalence, removing sinks creates a recollement situation.

\begin{corollary}\label{tt}
Let $E$ be a finite graph and $\overline E$  a graph where all sources and sinks of $E$ are removed. Then 
there is a recollement 
\begin{equation}\label{cgc}
\xymatrix{
L_K(\overline E) \rModd  \ar[r] & \ar@<3pt>[l] \ar@<-3pt>[l]  
L_K(E) \rModd  \ar[r] & \ar@<3pt>[l] \ar@<-3pt>[l]\bigoplus_{\text{sinks}} K  \rModd .}
\end{equation}
\end{corollary}
\begin{proof}
Let $H$ be the set of all sinks in $E$, which is a hereditary set. One can observe that 
$L(E_H)\cong  \bigoplus_{\text{sinks}} K.$ On the other hand, $E/\overline H$ is a graph by repeatedly removing all sinks from $E$ until there is no sinks exist. Indeed, if $u$ is a sink in 
$E/\overline H$, then $r(s^{-1}(u))$ has to be in $\overline H$, which gives $u \in \overline H$, as $\overline H$ is saturated. But this is a contradiction. 

By~\cite[Propositons~1.4 and 3.1]{abrams} removing sources from a graph, gives a Leavitt path algebra Morita equivalent to the original one. Thus repeating the source elimination give the Morita equivalence $L_K(\overline E) \approx L_K(E/\overline H)$. Now the corollary follows from Theorem~\ref{meis}. 
\end{proof}


\begin{example}
 
\hfill  
\begin{enumerate}
\item Let $T=L_K\big (\quad \, \, \xymatrix{\bullet \ar@(ld,lu)  \ar[r]  &  \bullet} \big)$ be the Topiliz algebra. By Corollary~\ref{tt} we have the following recollement situation

$$\xymatrix{
K[x,x^{-1}] \rModd  \ar[r] & \ar@<3pt>[l] \ar@<-3pt>[l]  
T \ar[r] & \ar@<3pt>[l] \ar@<-3pt>[l] K \rModd }$$

Although $T$ is indecomposable~\cite{gonzo}, but one can still break it down into less complicated algebras using recollement. 
\medskip

\item Let $T=L_K\big (\quad \, \, \xymatrix{\bullet \ar@(ld,lu)  \ar[r]  &  \bullet  \ar@(rd,ru) } \quad \, \,\big)$. Then we have the following recollement situation

$$\xymatrix{
K[x,x^{-1}] \rModd  \ar[r] & \ar@<3pt>[l] \ar@<-3pt>[l]  
T \ar[r] & \ar@<3pt>[l] \ar@<-3pt>[l] K[x,x^{-1}]\rModd  }$$

\medskip

\item In~\cite{aas}, Leavitt path algebras associated to finite acyclic graphs were classified. For such a graph, the Leavitt path algebra is a direct product of matrix algebras over the field $K$. Lemma~3.4 and Proposition~3.5 in \cite{aas} prove this by explicitly constructing an  isomorphism between these two algebras. In Corollary~\ref{tt}, for acyclic graph $E$, we get that $\overline E$ is in fact empty and thus the left hand side of (\ref{cgc}) is zero. Since in the general recollement situation~(\ref{popop}), $\mathscr A$ is a Serre subcategory of $\mathscr B$ and 
$\mathscr C \approx  \mathscr A / \mathscr B$, it follows that, in our setting (for $\mathscr B=0$) $L(E)$ is Morita equivalent to $\bigoplus_{\text{sinks}} K$, confirming the result of \cite{aas}.

\end{enumerate}
\end{example}

In a series of papers Rangaswamy~\cite{ranga1,ranga3}  and Ara-Rangaswamy~\cite{ranga2}  studied simple modules of Leavitt path algebras. In \cite{ranga2} they proved that for a Leavitt path algebra $L(E)$, where $E$ is a finite graph, any simple module is of the form of a Chen simple module (i.e., arising from an infinite path) if and only if any vertex of $E$ is the base of at most one cycle (these algebras have finite Gelfand-Kirillov dimension). 

Using the recollement and Kuhn's result \cite{kuhn} we give another characterisation of such algebras. In the following theorem we use a partial ordering defined in~\cite{ranga2}. A pre-order $\geq$ on the set of cycles of a directed graph $E$ is defined as follows: Let $c_1$ and $c_2$ be two cycles. We write $c_1 \geq c_2$ if there is a path from a vertex of $c_1$ to a vertex of $c_2$. 
For graphs whose vertices are bases of at most one cycle, $\geq$ is a partial order. In Theorem~\ref{minim} we work with minimal cycles with respect to $\geq$, i.e., cycles with no exist. 

In the following we will use Kuhn's result~\cite[Proposition~4.7]{kuhn}, which states that for a recollement 
$$\xymatrix{
\mathscr A \ar[r]|-{i} & \ar@<5pt>[l]^p    \ar@<-5pt>[l]_q  
\mathscr B \ar[r]|-{j}  & \ar@<5pt>[l]^r \ar@<-5pt>[l]_s   \mathscr C,}$$
there is a bijection between isomorphism classes of simple objects 
\begin{equation}\label{kjkj}
 \Big\{ \text{simples in } \mathscr A \Big\} \coprod  \Big \{ \text{simples in } \mathscr C \Big \} \longrightarrow \Big \{\text{simples in } \mathscr B  \Big\}. 
 \end{equation}

\begin{theorem}\label{minim}
Let $E$ be a finite graph whose vertices are base of at most one cycle. Then the isomorphism classes of simple (right) modules of $L_K(E)$ are in one to one correspondence with
the set of sinks and the set $A \times B$, where  $A$ is the set of all cycles in $E$ and $B$ is the set of all irreducible polynomials in $K[x,x^{-1}]$.

\end{theorem}
\begin{proof}
By Corollary~\ref{tt} (and its proof), we can write 
$$\xymatrix{
L_K(\overline E) \rModd  \ar[r] & \ar@<3pt>[l] \ar@<-3pt>[l]  
L_K(E) \rModd  \ar[r] & \ar@<3pt>[l] \ar@<-3pt>[l] \bigoplus_{\text{sinks}} K\rModd .}$$
where $\overline E$ is a graph with no sinks whose vertices are base of at most one cycle. 
Now by~(\ref{kjkj}) simple modules of $L_K(E)$ is in one to one correspondence with the sinks of $E$ and the simple modules of $L(\overline E)$. 

On the other hand $L(\overline E)$ can be {\it stratified} by cycles, i.e., we have iterated recollement situations 
$$\xymatrix{
B_i \ar[r] & \ar@<3pt>[l] \ar@<-3pt>[l]  
B_{i-1} \ar[r] & \ar@<3pt>[l] \ar@<-3pt>[l] C_i }$$ 
where $1\leq i \leq k$ and $k$ is the number of cycles in $E$.  
Here $B_i=L(b_i)$,  $C_i=L(c_i)$, where $b_0=\overline E$, $b_{i}=b_{i-1}/ \overline{c_{i}}$, and $c_i$ is a minimal cycle in the graph $b_{i-1}$ with respect to the ordering $\geq$. 

Again by~(\ref{kjkj}) there is a bijection between isomorphism classes of simple objects 
\begin{equation}\label{aoal}
 \Big\{ \text{simples in } B_{i} \Big\} \coprod  \Big \{ \text{simples in } C_{i} \Big \} \longrightarrow \Big \{\text{simples in } B_{i-1}  \Big\}. 
 \end{equation}

Since $c_i$ are cycles, $L(c_i)$ is Morita equivalent to the ring $K[x,x^{-1}]$. Thus the simple modules of $L(c_i)$ is in one to one correspondence with the irreducible polynomials of $K[x,x^{-1}]$. The theorem now follows from~(\ref{aoal})  by noting that $L(b_k)$ is also Morita equivalent to $K[x,x^{-1}]$.
\end{proof}

We note that the above theorem can also be deduced from the results of papers~\cite{ranga0,ranga2}. In \cite[Theorem~1.1]{ranga2},  a bijection between $S$, the set of non-isomorphic simple modules, and the set of primitive ideals of $L(E)$ was established. On the other hand Corollary 3.14 in~\cite{ranga0} implies there is a bijection between $S$ and $A \times B$, where $A$ is the set of all cycles in $E$ and $B$ is the set of all irreducible polynomials in $K[x,x^{-1}]$.


\begin{thebibliography}{99}


\bibitem{aas} G. Abrams, G. Aranda Pino, M. Siles Molina, \emph{Finite-dimensional Leavitt path algebras,} J. Pure Appl. Algebra {\bf 209} (2007), no. 3, 753--762.




\bibitem{abrams} G. Abrams, A. Louly, E. Pardo, C. Smith, \emph{Flow invariants in the classification of Leavitt path algebras,} J. Algebra {\bf 333} (2011)
202--23.

\bibitem{fuller} F. Anderson, K. Fuller, Rings and categories of modules, second ed. GTM 13, Springer Verlag, 1992.

\bibitem{gonzo} A. Aranda Pino, A. Nasr-Isfahani \emph{Decomposable Leavitt path algebras for arbitrary graphs}, Forum Math., In press (2015).


\bibitem{ranga2} P. Ara, K.M. Rangaswamy, \emph{Finitely presented simple modules over Leavitt path algebras}, J. Algebra {\bf 417} (2014), 333--352.


\bibitem{beil}  A. Beilinson, J. Bernstein, P. Deligne, \emph{Faisceaux Pervers (Perverse sheaves)}, in: Analysis and Topology on Singular Spaces, I,
Luminy, 1981, Asterisque {\bf 100} (1982) 5--171 (in French).


\bibitem{fran} V. Franjou, T. Pirashvili, \emph{Comparison of abelian categories recollements}, Doc. Math. {\bf 9} (2004) 41--56.

\bibitem{kuhn} N. Kuhn, \emph{Generic representations of the finite general linear groups and the Steenrod algebra: II}, K-Theory, {\bf 8} (1994) 395--428. 

\bibitem{psaer} C. Psaroudakis, \emph{Homological theory of recollements of abelian categories}, Journal of Algebra {\bf 398} (2014) 63--110. 
 
\bibitem{ranga0} K.M. Rangaswamy, \emph{ The theory of prime ideals of Leavitt path algebras over arbitrary graphs}, Journal of Algebra {\bf 375} (2013) 73--96.
 
 
\bibitem{ranga1} K.M. Rangaswamy, \emph{On simple modules over Leavitt path algebras}, J. Algebra {\bf 423} (2015), 239--258.

\bibitem{ranga3} K.M. Rangaswamy, \emph{Leavitt path algebras with finitely presented irreducible representations}, 
J. Algebra {\bf 447} (2016), 624--648.
 
 
\bibitem{wisba} R. Wisbauer, Foundations of modules and ring theory, 
CRC Press, 1991.
 
 
\end{thebibliography}
\end{document}